\documentclass[10pt,reqno, twoside]{amsart}
\usepackage{amsmath,amsthm,amssymb,amsfonts,ifpdf,bm}
\usepackage{enumitem}
\usepackage{graphicx}

\ifpdf
\usepackage[pdftex,pdfstartview=FitH,pdfborderstyle={/S/B/W 1},%
colorlinks=true, linkcolor=blue, urlcolor=blue, citecolor=blue,%
pagebackref=true]{hyperref}
\else
\usepackage[dvips]{hyperref}
\fi
\numberwithin{equation}{section}

\usepackage{tikz}
\usetikzlibrary{arrows}

\usepackage{graphicx}
\usepackage{epstopdf}

\usepackage{amscd}

\usepackage{graphics}
\usepackage{psfrag}
\usepackage{eucal}
\usepackage{mathrsfs}
\usepackage{color}

\usepackage{dsfont,bbm}
\usepackage{color}

\theoremstyle{plain}
\newtheorem{theorem}{Theorem}[section]
\newtheorem{proposition}[theorem]{Proposition}
\newtheorem{lemma}[theorem]{Lemma}

\theoremstyle{definition}
\newtheorem*{definition}{Definition}

\newtheorem*{claim}{Claim}

\newcommand{\C}{{\mathbb{C}}}

\newcommand{\R}{{\mathbb{R}}}
\newcommand{\Z}{{\mathbb{Z}}}
\newcommand{\N}{{\mathbb{N}}}
\newcommand{\Q}{{\mathbb{Q}}}
\newcommand{\T}{{\mathbb{T}}}

\newcommand{\GL}{\operatorname{GL}}
\newcommand{\SO}{\operatorname{SO}}

\newcommand{\dist}{\operatorname{dist}}

\newcommand{\Lip}{\operatorname{Lip}}
\newcommand{\D}{\operatorname{d}}
\newcommand{\m}{\mathbf{m}}

\newcommand{\overbar}[1]{\mkern 1.5mu\overline{\mkern-1.5mu#1\mkern-1.5mu}\mkern 1.5mu}

\newcommand{\bigslant}[2]{{\raisebox{.2em}{$#1$}\left/\raisebox{-.2em}{$#2$}\right.}}

\usepackage{graphicx}
\begin{document}

\title
[Simultaneous equidistributing and nondense points]
{Simultaneous equidistributing and nondense points for noncommuting toral automorphisms}

\subjclass[2010]{} \keywords{}

\author{Manfred Einsiedler}
\address{ETH Z\"{u}rich, Departement Mathematik, R\"{a}mistrasse 101, 8092 Z\"{u}rich Switzerland}\email{manfred.einsiedler@math.ethz.ch}
\author{Alex Maier}
\address{ETH Z\"{u}rich, Departement Mathematik, R\"{a}mistrasse 101, 8092 Z\"{u}rich Switzerland}\email{alex.maier@math.ethz.ch}

	\thanks{The authors~acknowledge the support by the SNF (Grant 200021-152819).}
	\date{Sept. 2015}
	\begin{abstract}
We show in prime dimension that for two non-commuting totally irreducible toral automorphisms the set of points
that equidistribute under the first map but have non-dense orbit under the second
has full Hausdorff dimension. In non-prime dimension the argument fails only if
the automorphisms have strong algebraic relations. 	\end{abstract}

\maketitle
%\tableofcontents

\section{Introduction}
%Introduction
%General Introduction

An important part of the theory of dynamical systems concerns itself with the behaviour of orbits.
In this paper we consider the structure of the set of points with prescribed orbit behaviour for quasi-hyperbolic automorphism $S: \T^d \longrightarrow \T^d$ 
of the $d$-dimensional torus. Taking a point $x \in \T^d$ its orbit $\{S^n x \colon n \in \N_0 \}$ 
can be dense or nondense in $\T^d$. The set of points with dense (non-dense) orbit we denote with $D(S)$
 (resp.\ $ND(S)$). Further we define the set $Eq(S) \subseteq D(S)$ of points whose orbits equidistribute 
on the torus with respect to the Lebesgue measure. 
It is wellknown that $Eq(S)$ has full measure (and so has $D(S)$) and that 
$ND(S)$ is winning (implying in particular that its Hausdorff dimension equals $\dim (ND(S)) = d$). 

Let us introduce a second automorphism $T: \T^d \longrightarrow \T^d$, it is trivial that 
$Eq(S) \cap Eq(T)$ has still full measure, and that $ND(S) \cap ND(T)$ is winning. 
But what can be said about $Eq(S) \cap ND(T)$? Conjecturally this set should be 
dense unless $S=T$ or other strong coincidences between $S$ and $T$ are satisfied. 

In \cite{BET} Bergelson, the first named author and Tseng showed that if $S$ and $T$ are commuting automorphisms of the torus which generate an algebraic $\Z^2$-action without rank one factors and $T$ is hyperbolic, then $\dim (D(S) \cap ND(T)) = d$.
Furthermore,  Lytle and the second named author \cite{LM} showed that for non-commuting $S$ and $T$ with $V_S^{0-} \oplus V_T^{0-} = \R^d$, where $V_S^{-0}$ (and $V_T^{-0}$) denotes the sum of the weak stable eigenspaces of $S$ (and $T$), the intersection has again full dimension. In this paper we want to show that either $\dim (Eq(S) \cap ND(T)) = d$ or that $S$ and $T$ satisfy a strong relationship that sometimes forces the two maps to commute. 

As above $\dim(\cdot)$  will always refer to the Hausdorff dimension.

\begin{theorem} \label{mainthm}
 Let $S,T$ be totally irreducible automorphisms of $\T^d$. Assume that the 
 weak stable subspaces for the two maps $V_{S}^{0-} \neq V_{T}^{0-}$ are different. 
 Then $$\dim (Eq(S) \cap ND(T)) = d.$$
\end{theorem}

As we discuss next the  case $V_{S}^{0-} = V_{T}^{0-}$ is quite special. 

\begin{theorem} \label{conjecture}
 Let $S,T$ be totally irreducible automorphisms of $\T^d$. Assume $V_{S}^{0-} = V_{T}^{0-}$ and that $\gcd(d, \dim V_S^{0-}) = 1$ (which holds e.g.\ if~$d$ is prime), then $S$ and $T$ commute.
\end{theorem}

In the case that $S$ and $T$ commute, we refer to the work \cite{BET} of Bergelson, the first named author and Tseng. We also give an example for two explicit $S$ and $T$ acting on $\T^4$ which have $V_{S}^{0-} = V_{T}^{0-}$ and do not commute but nonetheless have strong algebraic relationships. In this case we cannot say anything towards Theorem~\ref{mainthm}. 

\subsection{Outline of paper}

We start in section 2 by recalling some definitions and introducing the notion of $W$-dominating eigenspaces.

In section 3 we start with the assumption $V_{S}^{0-} \nsubseteq V_{T}^{0-}$ 
and show that on a line segment in some direction the set $ND(T)$ is winning,
which enables us to prove one part of the theorem along the lines of \cite{LM}.  

In section 4 we adapt an argument of Chaika and Eskin \cite{CE} to our setting. With this we can show that almost every point on a line segment pushed forward under an extension of $S$ gets invariant under some eigendirection $V$. With some extra work we can turn this invariance into equidistribution on $\T^d$ and conclude the proof using the same argument as in section 3. 

In the last section we show that in the case $V_{S}^{0-} = V_{T}^{0-}$ either $S$ and $T$ commute, or $S$ and $T$ have strong algebraic constraints, of which we give an example. In this example, $S$ and $T$ do not commute, but we have $V_{S}^{0-} = V_{T}^{0-}$. In this case our machinery cannot be applied, and we don't know anything about its orbit structure. Conjecturally the set $Eq(S) \cap ND(T)$ is still dense. Is it $d$-dimensional? It would be interesting to decide the question for this example.

\section{Preliminaries}

\begin{definition}
 A $d$-by-$d$ integer matrix $T$ is called \textit{irreducible} if its characteristic polynomial is irreducible over $\Q$. $T$ is called \textit{totally irreducible} if every power of $T$ is irreducible.
\end{definition}

Throughout the paper we let $T$ (and $S$) be a totally irreducible 
automorphism on $\T^d$ induced by an element of $\GL_d(\Z)$. 
Note that it follows immediately that every eigenvalue of~$T$ (and~$S$) 
has multiplicity equal to one and therefore $T$ and $S$ are diagonalizable over $\C$.
Depending on the context, sometimes $T$ is acting on the torus, and sometimes on $\R^d$. 
While working on the torus we might identify it with any convenient fundamental domain.

\subsection{Definitions, winning sets, and Hausdorff dimension}

We define the following sets related to the $\Z$-action of $T$ on $\T^d$:
\begin{align*}
ND(T) &= \left\{ x\in \T^d \colon \overbar{ \{T^nx\}_{n \in \N_0} } \subsetneq \T^d \right\}, \text{ and}\\ 
Eq(T) &= \left\{ \textstyle x \in \T^d \colon \tfrac{1}{N}\sum_{i=0}^{N-1} f ( T^n x ) \underset{N\to \infty}{\longrightarrow} \int f\, dm ~   \text{ for every } f \in C(\T^d) \right\}
\end{align*}
The first set is the set of points with nondense orbit under $T$, and the second one is the set of points whose $T$-orbits equidistribute.

We recall the notion of winning sets:
In \cite{S}, W. Schmidt introduced  the definition of winning along with the main properties of winning sets.  The game is played on $(X,\dist)$, a complete metric space.  Denote by $B(x,r)$ the closed metric ball around a point $x$ of radius $r$.  
The setup of the two player game is given by two parameters $0<\alpha,\beta<1$ and a set $S\subset X$. 
The game starts with round zero, in which one of the players, let's call him Bob, chooses a ball $B_0 = B(x_0, \rho)$ with $x_0 \in X$ and $0 < \rho \leq 1$.  The first round begins with the other player,  called Alice, choosing a point $y_1$, the center point of a ball $A_1 = B(y_1,\rho\alpha)\subset B_0$.  Bob chooses the next center point of a ball $x_1$ such that $B_1= B(x_1, \rho\alpha\beta) \subset A_1$.  This procedure is iterated with the $n$th round of the game beginning with Alice choosing a point $y_n$ with $A_n=B(y_n,\rho\alpha(\alpha\beta)^{n-1})\subset B_{n-1}$, and continuing with Bob choosing a point $x_n$ satisfying $B_n = B(x_n,\rho(\alpha\beta)^n) \subset A_n$.  At the end of the game we obtain
\[
 \textstyle\bigcap_{n\geq 1} A_n=\textstyle\bigcap_{n\geq 1} B_n=\{x_\infty \}.
\]
 If $x_\infty \in S$, then Alice wins.  If Alice can always find a winning strategy independent of the moves of Bob, the set $S$ is \textit{$(\alpha,\beta)$-winning}.  If there exists $\alpha$ such that $S$ is $(\alpha, \beta)$-winning for all $\beta>0$, then $S$ is an \textit{$\alpha$-winning} set.
As some of the main properties of $\alpha$-winning sets do not depend on the value of~$\alpha>0$,
we also speak simply of {\it winning} sets.

Winning sets have a number of useful properties, in particular in relationship with Hausdorff dimension.  W.~Schmidt showed in \cite{S} that winning sets within $\R^d$ have Hausdorff dimension $d$ (although more general statements exist~\cite{F1,F2}). 
Due to the fact that player Bob starts the game with an arbitrary ball it is clear that a winning set must be dense. In fact a winning set $S\subset\R^d$ is \textit{thick}, i.e.\ for every nonempty open set $U\subset\R^d$ we have $\dim(U \cap S) = \dim(S) = d$.

Another important feature of winning sets is the countable intersection property. As W. Schmidt
showed in \cite{S} for a countable collection $\{S_i\}$ of $\alpha_i$-winning sets with $\inf \alpha_i = \alpha_0 >0$, the intersection $\bigcap_i S_i$ is $\alpha_0$-winning. 

We are going to use Kleinbock and Margulis' version of the Marstrand Slicing Theorem \cite{KM}:

\begin{theorem}\label{thm:MarSlice}
Let $M_1$ and $M_2$ be Riemannian manifolds, $A \subset M_1$, $B \subset M_1 \times M_2$. Denote by $B_a$ the intersection of $B$ with $\{a\} \times M_2$ and assume that $B_a$ is nonempty for all $a \in A$. Then
$$\dim (B) \geq \dim (A) + \inf_{a \in A} \dim (B_a).$$
\end{theorem}

\subsection{$W$-Dominating eigenspaces} \label{sec:dominating}

Our methods rely on the decomposition of $\R^d$ into eigenspaces. 

In the case of complex eigenvalues we have to substitute them in the following way: Complex eigenvalues always appear in complex conjugated pairs, whose eigenvectors span a complex plane. Intersecting this plane with $\R^d$ leads to a real plane, of which we choose an orthonormal basis. In this sense a \emph{generalized eigenspace} will always mean a real eigenspaces respectively the real two-dimensional subspace corresponding to a pair of complex eigenvalues. We will denote a generalized eigenspace for $T$ by $E_T^{\nu}$,
where~$\nu$ is the eigenvalue.

By $V_T^{0},V_T^{+}$ and $V_T^{-}$ we refer to the sum of the generalized eigenspaces of $T$ which are central (with eigenvalues~$|\nu|=1$), unstable ($|\nu|>1$) respectively stable ($|\nu|<1$) for $T$. By $V_T^{0-} = V_T^{0} \oplus V_T^{-}$ we define the subspace which does not get expanded by $T$ and will refer to it as the \emph{weak stable subspace}. 

We fix a one-dimensional subspace $W \subset \R^d$ with $W \cap V_T^{0-} = \{0\}$ which is not contained in the weak stable subspace.

Taking any vector $w\in W \backslash \{0\}$, we can write it as the sum $w = \sum_{\nu} w_\nu$ of some generalized eigenvectors (some of them may get contracted or be central). Since $W$ is not completely contained in the weak stable subspace, there has to be some component $w_\nu \neq 0$ of an expanding eigendirection ($\lvert \nu \rvert > 1$) of $W$. Let $\nu_0$ be an eigenvalue with biggest absolute value $\lambda = \lvert \nu_0 \rvert$ among those eigenvalues $\nu$ with $w_{\nu} \neq 0$.

We write $w = w_{\lambda } + w_{<\lambda }$ where $w_{\lambda } = \sum_{\lvert \nu \rvert = \lambda } w_\nu$ and $w_{<\lambda } = \sum_{\lvert \nu \rvert < \lambda } w_\nu$. 
Furthermore we define the \textit{W-dominating~eigenspaces}
\[
 W_{\max} = \bigoplus_{\lvert \nu \rvert = \lambda, w_\nu \neq 0} E_T^{\nu},
\]
and the subspace
\[ 
 W_{<\max} = \bigoplus_{\lvert \nu \rvert < \lambda} E_T^{\nu}
\]
so that~$w_{\lambda}\in W_{\max}$ and~$w_{<\lambda}\in W_{<\max}$. 
Finally we let 
\[
 \eta = \max \{ \lvert \nu \rvert : \lvert \nu \rvert < \lambda \text{ and } w_\nu \neq 0 \}
\]
denote the second largest absolute value of an eigenvalue contributing to $W$.

\section{Winning}

Here we want to prove that the set of nondense points ND(T) is winning on any line segment not parallel to the weak stable subspace of $T$. This enables us to prove the following part of our main theorem:

\begin{theorem} \label{thm:win}
 Let $S,T$ be totally irreducible automorphisms of $\T^d$. Assume $V_{S}^{0-} \nsubseteq V_{T}^{0-}$. Then $\dim( Eq(S) \cap ND(T)) = d$.
\end{theorem}

\subsection{The winning property along line segments}

Let $T$ be a totally irreducible automorphism of the $d$-dimensional torus. 

We will use the supremum norm $\lVert \cdot \rVert = \lVert \cdot \rVert_\infty$ on $\R^d$ defined using the coordinates provided by the generalized eigenspaces.
Furthermore we normalize the norm such that the  ball of radius 2 is mapped injectively
and isometrically into $\T^d$.

Let $W$ be a one-dimensional subspace which is not contained in the weak stable subspace $V_T^{0-}$, i.e.\ $W \cap V_T^{0-} = \{0\}$ and vectors in~$W$ are eventually expanded (but $W$ may not be invariant under $T$). For $x \in \T^d$, let $A_{0}$ be an interval inside $x+W$ of finite length. We define 
\[
 ND_0(T)= \bigl\{y \in \T^d \mid 0 \notin \overbar{\{T^n y, n \in \N\}} \bigr\} \subseteq ND(T),
\]
i.e.~$ND_0(T)$ is the set of points whose forward orbit under~$T$ avoids $0$.

\begin{proposition}\label{prop:win}
 $ND_0 (T)\cap A_0$ is $\tfrac{1}{3}$-winning as a subset of $A_{0} \subset x+W$.
\end{proposition}

\begin{proof}
Taking any vector $w\in W$, we decompose it as in \S~\ref{sec:dominating} into the dominating and non-dominating part: 
I.e.~we write $w = w_{\lambda } + w_{<\lambda }$ where $w_{\lambda } = \sum_{\lvert \nu \rvert = \lambda } w_\nu$ and $w_{<\lambda } = \sum_{\lvert \nu \rvert < \lambda } w_\nu$. 
These vectors satisfy
 $$
 \lVert T^k w_{\lambda } \rVert = \lambda^k  \lVert w_{\lambda }\rVert \text{ and } \lVert T^k w_{<\lambda } \rVert \leq \eta^k \lVert w_{<\lambda }\rVert
$$
for all~$k\geq 0$, which leads to
$$\frac{\lVert T^k w_{\lambda } \rVert} {\lVert T^k w_{<\lambda } \rVert} \geq \left(\frac{\lambda } {\eta }\right)^k \frac{\lVert w_{\lambda }\rVert} { \lVert w_{<\lambda }\rVert}.
$$
Since the factor $\tfrac{\lambda} {\eta }>1$ is bigger than $1$, there exists an integer $k_I$ such that $\frac{\lVert T^{k} w_{\lambda } \rVert} {\lVert T^{k} w_{<\lambda } \rVert} \geq 1$ for all $k\geq k_I$. 
Since we are using the supremum norm, for $n\geq k_I$ we then have $\lVert T^n w \rVert=\lVert T^n w_{\lambda } \rVert$ and therefore 
\[
 \lVert T^{n+1} w \rVert = \lVert T^{n+1} w_{\lambda } \rVert = \lambda  \lVert T^{n} w_{\lambda } \rVert = \lambda  \lVert T^{n} w \rVert.
\]
 In other words: apart from the initial $k_I$ steps (where $k_I$ only depends on the direction of $W$), $w$ gets expanded by the factor $\lambda $ in each step.
We will use the metric induced by $\lVert \cdot \rVert_\infty$ on~$\R^d$ as above and define 
\[
 \dist_0(K) = \inf_{k \in K ,n\in\Z^d} \lVert k+n \rVert_\infty
\]
for any subset~$K\subset\R^d$ or $K\subset\T^d$.

\emph{The game.}
We are now going to discuss the Schmidt game for the proof that the 
set $ND_0(T) \cap A_0$ is ($\alpha, \beta$)-winning with $\alpha = \tfrac{1}{3}$ 
and for any $\beta \in (0,1)$.  
Let player B start with choosing a ball $B_0 = B(y_0,\rho)$ inside $A_0$. Without loss of generality we may 
assume $\rho$ is quite small. In fact, we may assume $\rho$ is so small so
 that the diameter of~$T^{k_I}B_0$ has diameter less than~$\lambda^{-1}$. 
For otherwise we may simply apply sufficiently many steps of the game without any particular strategy of A
(which shrinks the interval by a factor~$\alpha\beta$ each time) 
and pretend that the resulting ball was what player B chose initially. 

Given $B_0$ (with this property) we define~$k_0$ to be the largest integer so 
that $T^{k_0} B_0$ has diameter smaller than~$1$. Note that $k_0> k_I$
and that the above discussion implies that~$T^{k_0} B_0$
has diameter at least $\lambda^{-1}$.

We now assume that the game has already been played for $n$ steps, that nested balls
$B_0\supseteq A_0\supseteq B_1\supseteq\cdots\supseteq B_n$ have been chosen,
and that an increasing sequence of integers $k_0\leq k_1\leq \cdots\leq k_{n-1}$
has been constructed inductively. 

\emph{The strategy.}
Given the ball $B_n$  
player A cuts $B_n$ into three subintervals $I_1,I_2,I_3$ 
of equal size and since~$\alpha=\frac13$ player A is allowed
to choose any of the three pieces as the next move $A_n$. 
We define~$k_n$ to be the largest integer so 
that $T^{k_n} B_n$ has diameter smaller than~$1$. Once more 
this implies that~$T^{k_n} B_n$
has diameter at least $\lambda^{-1}$. Now player A
considers the sets $T^{k_n}I_1,T^{k_n}I_2,T^{k_n}I_3$ that together
trisect a line segment within $\T^d$ of size between~$\lambda^{-1}$ and $1$.
Even in the worst case (when $0$ belongs to the center of the second interval)
one of the three pieces, say $I_\ell$, satisfies
\[
  \dist_0(T^{k_n}I_\ell)\geq \tfrac16\lambda^{-1}.
\] 
Player A chooses one such intervals out of the three.

\emph{Winning property.}
Suppose that the game has run its course and we have found the point $z\in\bigcap_nA_n$
and the sequence of integers $k_0\leq k_1\leq\cdots$ by following the above strategy.
We will now show that $0\notin\overbar{\{T^kz\mid k\geq 0\}}$.
First notice that it is clear that 
$T^{k_n}z$ has distance $\geq\frac16\lambda^{-1}$ from $0$ for all $n\geq 0$ and we only have to 
worry about the powers $T^kz$ with $k$ not being of the form $k=k_n$ for some $n$.
There are only finitely many integers~$k\in[0, k_0)$ and since~$T^{k_0}z\neq 0$ and $0$ is a fixed point
these points are not an issue for the desired conclusion. For the remaining integers $k>k_0$ it is important
to note that the sequence~$k_{n+1}-k_n$ is bounded (where the bound only depends on the parameters~$\alpha\beta$ and~$\lambda$). Together with continuity of $T$ this shows that~$z\in ND_0(T)$,
the winning property, and the proposition.
\end{proof}

\subsection{Proof of Theorem~\ref{thm:win}}
%Here we proof how it follows from winning on a line to full Hdim

For the proof we are going to use that the orbits of two points on the same weak stable manifold have the same behaviour:

\begin{lemma} \cite[Lemma 4.3 and Lemma 4.7]{LM} \label{lem:foliation}
 Let $x \in \T^d$ and let $v\in V_S^{0-}$.  Then $x$ equidistributes under $S$ if and only if $x+v$ equidistributes under $S$. Furthermore we have $x \in ND(S)$ if and only if $x+v \in ND(S)$.
\end{lemma}

As a corollary of this and ergodicity of~$S$ we have the following lemma:
\begin{lemma}\label{lem:measure}
 Let $W \subseteq V_S^{0-}$ be a 1-dimensional subspace, let $W^{\bot}$ be any subspace of $\R^d$ with $W \oplus {W^{\bot}} = \R^d$ and $\m_{W^{\bot}}$ be the Lebesgue measure defined on ${W^{\bot}}$. Then 
for any~$x_1\in\T^d$ we have
\[
 \m_{W^{\bot}} (\{v\in {W^{\bot}}\mid x_1+v\notin Eq(S)\}) = 0.
\]
\end{lemma}

\begin{proof}
	 In fact, Lemma~\ref{lem:foliation} shows that for any $v \in {W^{\bot}}$ and $w \in W$ we have $x_1+v \in Eq(S) \Leftrightarrow x_1+v + w \in Eq(S)$. This means that the set $Eq(S)$ equals a union of $W$-cosets.  Ergodicity for~$S$ and Fubini's theorem now gives the lemma.
\end{proof}

\begin{proof}[Proof of Theorem~\ref{thm:win}]
Since $V_{S}^{0-} \nsubseteq V_{T}^{0-}$ we can choose a one-dimensional subspace $W \subseteq V_{S}^{0-}$ such that $W \cap V_{T}^{0-} = \{0\}$. Choose ${W^{\bot}}$ to be any subspace of $\R^d$ such that ${W^{\bot}}\oplus W = \R^d$.
We choose the same norm as in the proof about winning, i.e.\ the supremum norm with respect to the eigenvectors of $T$.

Let $x_1 \in \T^d$ and~$\epsilon\in(0,\frac18)$ be arbitrary. We want to show that $Eq(S) \cap ND(T)$ is thick, i.e.\ its Hausdorff dimension in $B(x_1, 2\epsilon)$ is $d$. Applying Lemma~\ref{lem:measure} we obtain that the set $ \{v\in W^{\bot}\mid x_1+v\in Eq(s)\}$ has full measure as a subset of ${W^{\bot}}$ and in particular 
\begin{equation}\label{largedim}
\dim \bigl(({x_1+(B(0,\epsilon)\cap W^{\bot}}))\cap Eq(S) \bigr)= d-1.
\end{equation}

For any~$x$ let $A_0(x) = x+(B(0,\epsilon) \cap W)$ be the interval of length ${2}\epsilon$ inside $x+W$ containing $x$  to which we can apply Proposition~\ref{prop:win}. Since $ND_0(T) \subseteq ND(T)$ we get 
that for any~$x \in \T^d$ the set $ ND(T) \cap A_0(x)$ is winning as a subset of $A_0(x)$. 

Now we are ready to apply the Marstrand Slicing Theorem (Theorem~\ref{thm:MarSlice}) by setting
\begin{align*}
 M_1 &= {x_1+(B(0,\epsilon)\cap W^{\bot})},\\
 M_2 &= W \cap B(0,\epsilon),\\
A &= M_1\cap Eq(S),\text{ and } \\
B &=  ND(T) \cap Eq(S) \cap (M_1+M_2),
\end{align*}
where we identify~$M_1\times M_2$ with~$M_1+ M_2\subset B(x_1,{2\epsilon})$. 
Then~$\dim A=d-1$ by~\eqref{largedim} and for any~$x\in A$ we have
\[ 
 B_x=B\cap (x+M_2)= ND(T) \cap (x+(W \cap B(0,\epsilon)))
\]
by Lemma~\ref{lem:foliation}. However, since $ND(T) \cap A_0(x)$ is winning and hence thick
as a subset of $A_0(x)$, the same holds for $B_x$ as a subset of $x+M_2$ and it follows that~$\dim B_x=1$ for all~$x\in A$. 
Theorem~\ref{thm:MarSlice} now implies
\[
 \dim (ND(T)\cap Eq(S)\cap B(x_1,2\epsilon))=d
\]
as required.
\end{proof}

\section{Equidistribution}

Here we want to prove that the set of equidistributing points $Eq(S)$ has full measure on any line not parallel to the weak stable subspace of $S$. Then we will prove the second part of our theorem, which is the following theorem:

\begin{theorem} \label{thm:meas}
 Let $S,T$ be totally irreducible automorphisms of $\T^d$. Assume $V_{T}^{0-} \nsubseteq V_{S}^{0-}$. Then $\dim Eq(S) \cap ND(T) = d$.
\end{theorem}

The structure of this section is as follows: 

First we extend the torus and the action of $S$ to $\tilde{S}$ acting on $K \ltimes \T^d$ for some compact group $K$. With this we can change $S$ from having complex eigenvalues to having positive ones, but acting on a bigger space.

Then we adapt an argument of Chaika and Eskin \cite{CE} to our setting. Chaika and Eskin showed on $G / \Gamma$ that a point measure on the orbit of $\SO(2)$ averaged under the geodesic flow is for almost all points on the $\SO(2)$ orbit in the limit also invariant under some unipotent direction. 
Here we show that for almost all points on a line the dirac measure of this point gets invariant under some eigendirection of $\tilde{S}$ when being pushed forward with $\tilde{S}$.

Subsequently we make a change of coordinates to get some invariance of our original map $S$ acting on $X$. Using Poincare recurrence we then can show that for almost every point on a line parallel to $W$ its point measure equidistributes when being pushed forward by $S$.
This enables us to prove the main theorem of this section.

\subsection{Extending the torus and other preparations for the complex case}
% Here we describe the additional things arising from complex eigenvalues

Let $S$ be a totally irreducible automorphism of the $d$-dimensional torus. So $S$ is also an automorphism of $\R^d$ on which we will again work with the supremum norm (aligned with the 
generalized eigenspaces of ${S}$).

Let $W \nsubseteq V_S^{0-}$ be a one-dimensional subspace of $\R^d$ with $ W_{\max}$ being its $W$-dominating generalized eigenspace (see Section~\ref{sec:dominating}) with absolute value of the eigenvalues equal to $\lambda > 1$. 

In the case the $W$-dominating eigenvalues are positive, the following construction becomes a lot easier: one then can define $k_S = \mathbbm{1}$ and $K=\{\mathbbm{1}\}$ will be the trivial group, giving $S = \tilde{S}$,  $K \ltimes \R^d = \R^d$, and~$V=W_{\max}$ (which is one-dimensional in this case). In the general case we are going to need more notation which we will introduce now.

For each pair $\{v_{\lambda_j},\overbar{v_{\lambda_j}}\}$ of complex conjugated eigenvectors we define a rotation $K_{\lambda_j}$ as follows: $K_{\lambda_j} v_{\lambda_j} = \frac{\overbar{\lambda_1}}{\lvert \lambda_j \rvert} v_{\lambda_j}$ and $K_{\lambda_j} v_{\overbar{\lambda_j}} = \frac{\lambda_1}{\lvert \lambda_j \rvert} v_{\overbar{\lambda_j}}$ and the identity on the other eigenvectors. We note
that this defines a real matrix that rotates the corresponding generalized eigenspace (with the angle of rotation being the opposite to the angle of rotation for $S$).
For each negative eigenvalue $\lambda_j$ we define $K_{\lambda_j} v_{\lambda_j} = - v_{\lambda_j}$ and the identity on the other eigenvectors.
Let $k_S = \prod_j K_{\lambda_j}$ be the product of all those rotations, where the product is taken over all negative eigenvalues $\lambda_j$ and all pairs of complex conjugated eigenvalues $\{\lambda_j,\overbar{\lambda_j}\}$.

We define the linear map $S_{\mathrm{pos}} = k_S  S=S k_S$ on $\R^d$.
This map has only positive eigenvalues. Some of them might occur with higher geometric multiplicity (in particular when $S$ has several eigenvalues of the same absolute value). 
In particular $W_{\max}$ may have dimension bigger than one. We choose $V$ to be the real one-dimensional eigenspace spanned by $w_{\lambda}$, where $w_{\lambda } = \sum_{\lvert \nu \rvert = \lambda } w_\nu$ is as defined in Section~\ref{sec:dominating}. Note that $w_{\lambda}$ is indeed an eigenvector of~$S_{\mathrm{pos}}$ but in general not of~$S$.

We may refer to~$V$ as the $W$-dominating eigenspace with respect to $S_{\mathrm{pos}}$, which is a one-dimensional eigenspace of $S_{\mathrm{pos}}$ (in the strict sense) and gets expanded by the factor $\lambda > 1$. However, we note
that on the other hand $S_{\mathrm{pos}}$ does in general no longer preserve the lattice $\Z^d$ and hence does not
define a map on~$\T^d$. Therefore we have to extend the torus in the following way:

We define $K \subseteq \GL_d(\R)$ to be the topological group generated by the rotation $k_S$ and notice that $K$ is a compact abelian group.
$S$ and $K$ commute, since they have common eigenspaces.

Both of them act on $\R^d$, so that we can consider the semidirect product
$$  {\langle S \rangle  K \ltimes \R^d}
$$
where we define the multiplication rule 
\[
(S^n  k_1,x) (S^m  k_2, y)=(S^{n+m}  k_1 k_2, x + S^n k_1 y)
\]
 inspired by the usual matrix multiplication 
$\begin{pmatrix}
k_1 & x\\
0 & 1
\end{pmatrix} \begin{pmatrix}
k_2 & y\\
0 & 1
\end{pmatrix} = \begin{pmatrix}
k_1k_2 & k_1y + x\\
0 & 1
\end{pmatrix}$. 
We define $\Gamma = \langle S \rangle \ltimes \Z^d $ and  
$$  X = \bigslant{\langle S \rangle  K \ltimes \R^d}{\langle S \rangle \ltimes \Z^d} 
\cong (K \ltimes \R^d) / \Z^d,$$
where the latter isomorphism is useful for understanding $X$
and the original definition gives us a way of letting $S$ act on $X$.

On $X$ we define a new transformation $\tilde{S}$ by the following formula:
\begin{align*}
	 \tilde{S} (k,x) \Gamma &= (Sk_S,0)(k,x) \Gamma \\
&= 
(S k_S k ,S_{\mathrm{pos}}x)
( S^{-1} , 0) \Gamma = (k_S k, S_{\mathrm{pos}}x) \Gamma 
\end{align*}  
We see that on the second coordinate $S_{\mathrm{pos}}$ is acting on $\R^d$. On the first coordinate we keep track which complex rotation was used to ``straighten
the map $S$ to the map $S_{\mathrm{pos}}$''. 

In order to understand the geometry of $X$ better let us make a few remarks. For any $v\in
\R^d$ the left translation action on our semi-direct product is given simply by
\[
 (\mathbbm{1},v) (k,x)=(k,x+v),
\]
but for any $n\in\Z^d$ the identification modulo $(\mathbbm{1},n)$ (i.e.\ right translation) has the form
\[
 (k,x)(\mathbbm{1},n)=(k,x+kn),
\]
i.e.\ the coordinate $k$ amounts to a rotation of the lattice $\Z^d$.
Finally we claim that
\[
 \Theta:  (k,x)\Gamma\in X\mapsto k^{-1}x+\Z^d\in\T^d
\]
defines a factor map between the transformation $\tilde{S}$ on $X$ and $S$ on $\T^d$. 
Indeed 
\[
 \Theta\bigl((k,x)(\mathbbm{1},n)\Gamma\bigr)=
 \Theta\bigl((k,x+kn)\Gamma\bigr)=k^{-1}(x+kn)+\Z^d=\Theta\bigl((k,x)\Gamma\bigr)
\]
for all $k\in K$, $x\in\R^d$, and $n\in\Z^d$, shows that $\Theta$ is well-defined,
and
\begin{multline*}
  \Theta\bigl(\tilde{S}((k,x)\Gamma)\bigr)=
  \Theta\bigl((k_Sk,S_{\mathrm{pos}}x)\Gamma\bigr)=\\
  k_S^{-1}S_{\mathrm{pos}}k^{-1}x+\Z^d=Sk^{-1}x+\Z^d=S\Theta\bigl((k,x)\Gamma\bigr)
\end{multline*}
for all $k\in K$ and $x\in\R^d$ shows that $\Theta\circ\tilde{S}=S\circ\Theta$. 

\subsection{The Chaika-Eskin  Argument} \label{sec:ChaikaEskin}

In this subsection we adapt the argument by Chaika and Eskin \cite{CE} to our setting. 

The above construction lets us work with the following: We started with $W \nsubseteq V_S^{0-}$, a one-dimensional subspace of $\R^d$ with $V \subseteq W_{\max}$ it's $W$-dominating generalized eigenspace
with respect to $S_{\mathrm{pos}}$. 
Also recall that $\lambda >1$ is the eigenvalue for~ $S_{\mathrm{pos}}$ on~$V$. 
Using the constructed $k_S$ we got $S_{\mathrm{pos}} = S k_S$ which is now an automorphism of $\R^d$ with only positive eigenvalues. Then we constructed the bigger space $X = (K \ltimes \R^d) / \Z^d$ and act from the left with 
$ \tilde{S} $ in the following way:
$ \tilde{S} (k,x) \Gamma = (k_S k, S_{\mathrm{pos}}x) \Gamma $.

Of course we need a metric on the space $X$, which we define as follows: On  $\R^d$ we will use the supremum norm $\left\| \cdot \right\|_\infty$ that we defined earlier. As $K$ is a subset of $(\mathbb{S}^1)^{\times d}$    we take a rotation invariant metric on $(\mathbb{S}^1)^{\times d}$
and restrict it to $K$. The metric~$\D$ on $X$ is defined as the product metric. In particular 
$\D \bigl((k,x),(k,y)\bigr) =  \left\| x - y \right\|_\infty$. 

Let us write $W$ as an additive $1-$parameter group $\{ w_\varphi=\varphi w : \varphi \in \R\}$. To simplify notation, we assume~$\|w\|_\infty=1$ so that $\lvert \varphi \rvert = \lVert w_\varphi \rVert_\infty$. 
We will frequently work with the set $\{w_\varphi : \lvert \varphi \rvert \leq \frac{1}{2} \}$.

For a fixed $x \in X$ and $\lvert \varphi \rvert \leq \frac{1}{2}$ we define the sequence of measures 
\begin{equation}\label{muN}
\mu_{N,\varphi} = \tfrac{1}{N} \sum_{n=0}^{N-1} \tilde{S}^n_* \delta_{(\mathbbm{1}, w_\varphi)x},
\end{equation}
for any $N \in \N$.

Now we are ready to state the main technical result of the section. 

\begin{proposition}\label{prop:invariance}
	For almost every $\varphi \in [-\frac{1}{2},\frac{1}{2} ]$ 
 any weak$^*$ limit $\mu = \mu_{\varphi}$ of 
 the sequence $\mu_{N,\varphi} $ is invariant under  $V$. %[and invariant under the action of $\tilde{S}$].
\end{proposition}

We note that since the group of elements that preserve a probability measure on a compact space is automatically closed, it suffices to prove the invariance under $v_a$ for any fixed $v_a \in V$.  
To see this e.g.\ choose two elements $v_a$ and $v_b$ that are linearly independent over $\Q$.

Fix some $v_a \in V$.
Fix some $\varphi$, then the invariance of any weak$^*$ limit as in Proposition~\ref{prop:invariance} is equivalent to the following. 
For any %(or any member of a dense countable family of) 
Lipschitz function $\phi \in \Lip(X)$ we have to show
\begin{equation}\label{mainconvergence}
 \frac{1}{N} \sum_{n=0}^{N-1} \phi \big( 
 (\mathbbm{1},v_a)\tilde{S}^n(\mathbbm{1},w_\varphi)x\big) - \phi \big( \tilde{S}^n(\mathbbm{1},w_\varphi)x\big) \longrightarrow 0 \text{ as } N \rightarrow \infty.
\end{equation}
We note that
\begin{equation}\label{seimpleeq}
(\mathbbm{1},v_a)\tilde{S}^n(\mathbbm{1},w_\varphi)x=
(\mathbbm{1},v_a+S_{\mathrm{pos}}^nw_\varphi)\tilde{S}^nx.
\end{equation}
As before we decompose $w$ into $w= w_{\max} + w_{< \max}$, the first being the projection of $w$ onto $V$. 
Applying $S_{\mathrm{pos}}$ gives $S_{\mathrm{pos}}^n w = \lambda^n w_{\max} + S_{\mathrm{pos}}^n w_{< \max}$. As before let $\eta$ be the second largest eigenvalue occuring in $W$ (see Section~\ref{sec:dominating}). 
Since no eigenvector contributing to $W_{<\max}$ can have an eigenvalue of absolute value bigger than $\eta$ we have $\lVert S_{\mathrm{pos}}^n w_{< \max} \rVert \leq \eta^n \lVert w_{<\max} \rVert$.

For the upcoming bounds we will use the notation $a \ll b$ meaning there is a constant $C = C(W,\phi, v_a....)$ such that $a \leq C b$.
Let $t_n\in\R$ be such that $\lambda^n t_n w_{\max} = v_a$, which implies $|t_n| \ll \lambda^{-n}$. We note that the sequence $t_n$ is independent of $\phi$, $x$ and $\varphi$ and satisfies
\[
 \lVert S_{\mathrm{pos}}^n t_n w - v_a \rVert = \lVert t_n S_{\mathrm{pos}}^n w_{< \max} \rVert \leq \eta^n \cdot \lvert t_n \rvert \ll \left(\frac{\eta}{\lambda}\right)^n.
\]

%\begin{figure}[h]
%\psfrag{W}{$W$}\psfrag{w_{t_n}}{$w_{t_n}$}\psfrag{V}{$V$}\psfrag{v_a}{$v_a$}\psfrag{S_{real}^nW}{$S_{\mathrm{pos}}^n W$}\psfrag{S_{real}^n w_tn}{$S_{\mathrm{pos}}^n w_{t_n}$}
%\includegraphics{tn.eps}
%\caption{Definition of $t_n$, with $v_a = \lambda_n t_n w_{\max}$.}
%\end{figure}

\begin{figure}[ht]

\begin{tikzpicture}[line cap=round,line join=round,>=triangle 45,x=1.5cm,y=1.7cm]
\clip(-0.41445383543823296,-0.3361344932913753) rectangle (7.065903650627493,3.1680221995257063);
\draw [domain=-0.31445383543823296:7.065903650627493] plot(\x,{(-0.--0.5122117405247799*\x)/0.21670825636195187});
\draw [domain=-0.31445383543823296:6.265903650627493] plot(\x,{(-0.--0.35845429289841046*\x)/5.388549676521646});
\draw [domain=-0.31445383543823296:6.265903650627493] plot(\x,{(-0.--0.8*\x)/5.38});
\draw (0.7059364988094905,3.005647345857368) node[anchor=north west] {$W$};
\draw (0.3425098044198904,0.651880987288011) node[anchor=north west] {$w_{t_n}$};
\draw (6.227226663574569,1.2266902756356426) node[anchor=north west] {$S^n_{\mathrm{pos}}W$};
\draw (4.9831891327794,1.3347564262829964) node[anchor=north west] {$S^n_{\mathrm{pos}}w_{t_n}$};
\draw (5.262748128463707,0.23265274484047185) node[anchor=north west] {$v_a$};
\draw (6.394962060985154,0.5960794392300723) node[anchor=north west] {$V$};
%\begin{scriptsize}
\draw [fill=black] (0.21670825636195187,0.5122117405247799) circle (1.5pt);
\draw [fill=black] (5.388549676521646,0.35845429289841046) circle (1.5pt);
\draw [fill=black] (5.38,0.8) circle (1.5pt);
%\end{scriptsize}
\end{tikzpicture}
\caption{Definition of $t_n$, with $v_a = \lambda_n t_n w_{\max}$.}
\end{figure}
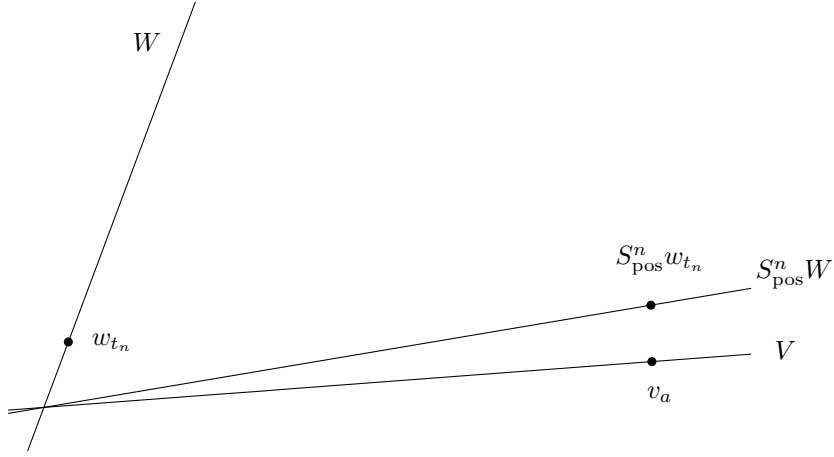

With this and \eqref{seimpleeq} we have 
$$
 \D\Bigl(  (\mathbbm{1},v_a)\tilde{S}^n(\mathbbm{1},w_\varphi)x,
 \tilde{S}^n(\mathbbm{1},w_{\varphi+t_n})x\Bigr) = \lVert S_{\mathrm{pos}}^n(w_{t_n} ) - v_a\rVert \ll \left(\frac{\eta}{\lambda}\right)^n
$$ 
and so
\[
	 \bigl\vert \phi\bigl( (\mathbbm{1},v_a)\tilde{S}^n(\mathbbm{1},w_\varphi)x\bigr) -  
	 \phi\bigl(\tilde{S}^n(\mathbbm{1},w_{\varphi+t_n})x\bigr)\bigr\vert \ll  
 \left(\frac{\eta}{\lambda}\right)^n.
\]
Therefore,
$$ 
\frac{1}{N} \sum_{n=0}^{N-1} \bigl\vert \phi\bigl((\mathbbm{1},v_a)\tilde{S}^n(\mathbbm{1},w_\varphi)x\bigr)  -  
\phi\bigl(\tilde{S}^n(\mathbbm{1},w_{\varphi+t_n})x\bigr)\bigr\vert \ll \frac{1}{N} \sum_{n=0}^{N-1} \left(\frac{\eta}{\lambda}\right)^n \ll \frac{1}{N} \longrightarrow 0
$$
as $ N \rightarrow \infty$
and we see that \eqref{mainconvergence} will follow once we show that
$$ \frac{1}{N} \sum_{n=0}^{N-1} \phi\bigl(\tilde{S}^n(\mathbbm{1},w_{\varphi+t_n})x \bigr) -  \phi\bigl(\tilde{S}^n(\mathbbm{1},w_{\varphi})x\bigr)
 \longrightarrow 0 \text{ as } N \rightarrow \infty .$$
This is the content of the next lemma whose proof will occupy the remainder of the subsection.

\begin{lemma}\label{lem:limit0}
For any $\phi\in \Lip(X)$ and a.e.\ $\varphi \in [-\frac{1}{2},\frac{1}{2}]$ we have
$$ \lim_{N\rightarrow \infty} \frac1N\sum_{n=0}^{N-1} f_n(\varphi) \longrightarrow 0,$$
where we define for every~$n\geq 1$ the real valued functions
\begin{align*}
	f_n(\varphi) &= \phi\bigl(\tilde{S}^n(\mathbbm{1},w_{\varphi+t_n})x \bigr) -  \phi\bigl(\tilde{S}^n(\mathbbm{1},w_{\varphi})x\bigr)\\
	&=\phi\bigl((\mathbbm{1},S_{\mathrm{pos}}^nw_{\varphi+t_n})\tilde{S}^nx \bigr) -  \phi\bigl((\mathbbm{1},S_{\mathrm{pos}}^nw_{\varphi})\tilde{S}^nx\bigr)
\end{align*} 
on $[-\frac{1}{2},\frac{1}{2}]$.
\end{lemma}

The following shows that the functions $f_n$ are nearly independent of each other.

\begin{lemma}\label{lem:integral}
Let $m,n \in \N$, $\phi \in \Lip(X)$. Then we have
 $\Bigl\lvert\int_{-\frac{1}{2}}^{\frac{1}{2}} f_n(\varphi)f_m(\varphi) d\varphi \Bigr\rvert \ll \lambda^{-\frac{\lvert n-m \rvert}{2}} $.
\end{lemma}
\begin{proof}
 Fix $m < n$ and define $\delta_0 = \lambda^{-\frac{\lvert n-m \rvert}{2}} \in (0,1)$ and $\delta = \lambda^{-m} \delta_0 =  \lambda^{-\frac{m+n}{2}}$. 

Using $\lvert t_n \rvert \ll \lambda^{-n}$ we start by estimating the following average over the set $A_\varphi = [\varphi - \delta, \varphi + \delta]$:
\begin{align*}
\frac{1}{\lvert A_\varphi \rvert} \biggl\lvert \int_{A_\varphi} f_n(\theta) d\theta \biggr\rvert 
 =& \frac{1}{\vert A_\varphi \rvert} \biggl\lvert \int_{A_\varphi} 
\phi\bigl(\tilde{S}^n(\mathbbm{1},w_{\theta+t_n})x \bigr) -  \phi\bigl(\tilde{S}^n(\mathbbm{1},w_{\theta})x\bigr) 
d\theta \biggr\rvert \\
\leq& \frac{1}{\vert A_\varphi \rvert} \int_0^{t_n} \bigl\lvert \phi\bigl(\tilde{S}^n(\mathbbm{1},w_{\varphi-\delta + \theta})x ) \bigr\rvert + \bigl\lvert \phi\bigl(\tilde{S}^n(\mathbbm{1},w_{\varphi+\delta + \theta})x ) \bigr\rvert d\theta \\
 \leq& \frac{1}{\vert A_\varphi \rvert} 2 \lvert t_n \rvert \cdot \sup(\phi) 
\ll \frac{\lambda^{-n}}{\vert A_\varphi \rvert} < \delta_0,
\end{align*}
where we assume $t_n > 0$ (in the case $t_n < 0$ the integral goes over $[t_n,0]$).

Then we fix $\theta \in A_\varphi$ and have 
\begin{align*}
 \lvert f_m(\varphi) - f_m(\theta) \rvert 
=& \bigl\lvert \phi\bigl((\mathbbm{1},S_{\mathrm{pos}}^m w_{\varphi + t_m})\tilde{S}^m x\bigr) 
  - \phi\bigl((\mathbbm{1},\lambda^m w_{\theta-\varphi})(\mathbbm{1},S_{\mathrm{pos}}^m w_{\varphi + t_m})\tilde{S}^m x\bigr) \\
& - \phi\bigl((\mathbbm{1},S_{\mathrm{pos}}^m w_{\varphi})\tilde{S}^m x\bigr) 
  + \phi\bigl((\mathbbm{1},\lambda^m w_{\theta-\varphi})(\mathbbm{1},S_{\mathrm{pos}}^m w_{\varphi})\tilde{S}^m x\bigr) \bigr\rvert\\
 \leq& 2 \Lip(\phi) \lambda^m \lVert w_{\theta - \varphi} \rVert 
 \ll \lambda^m \lvert \theta - \varphi \rvert 
 \leq \lambda^m \delta = \delta_0. %\\
\end{align*}
In other words, $f_m$ is almost constant on small intervalls of length $O\Bigl(\lambda^{-\frac{\lvert n-m \rvert}{2}}\Bigr)$, which we use in the following step: 

\begin{align*}
\frac{1}{\vert A_\varphi \rvert} \biggl\lvert \int_{A_\varphi} f_m(\theta) f_n(\theta) d\theta \biggr\rvert 
=& \frac{1}{\vert A_\varphi \rvert} \biggl\lvert \int_{A_\varphi} ( f_m(\varphi) + O(\delta_0)) \cdot f_n(\theta) d\theta \biggr\rvert\\
\leq& \lvert f_m(\varphi) \rvert \cdot  \frac{1}{\vert A_\varphi \rvert}  \biggl\lvert \int_{A_\varphi} f_n(\theta) d\theta \biggr\rvert +O(\delta_0) 
\ll \delta_0,
\end{align*}
where we used $\lvert f_m(\varphi) \rvert \ll 1$
 and
in the last step the estimate we computed right above.

Now we cover the interval $[-\tfrac12,\tfrac12]$ with intervals of the form $A_{\varphi_i}$ and an inverval $B$ such that $[-\frac{1}{2},\frac{1}{2}] = B \sqcup \bigsqcup_{i} A_{\varphi_i}$ and $ \lvert B \rvert < 2\delta \ll \delta_0$. With
\begin{align*}
\biggl\lvert\int_{-\frac{1}{2}}^{\frac{1}{2}} f_n(\varphi)f_m(\varphi) d\varphi \biggr\rvert
\leq& 2\delta \lVert\varphi\rVert_\infty + \sum_{i} \vert A_{\varphi_i} \rvert \cdot \frac{1}{\vert A_{\varphi_i} \rvert} \cdot \biggl\lvert \int_{A_{\varphi_i}} f_n(\theta) f_m(\theta) d\theta \biggr\rvert  \\
\ll& \delta_0 + \sum_{i} \lvert A_{\varphi_i} \rvert \cdot \delta_0 \ll \delta_0
\end{align*}
we conclude the proof of the lemma.
\end{proof}

Recalling the following standard results from measure theory we are ready for the proof of Lemma~\ref{lem:limit0}.

\begin{lemma}[Chebyshev inequality]\label{Chebyshev}
Let $g: \Omega \rightarrow \R$ be measurable, $\int_\Omega g^2 d\m \leq D$. Then for all $s >0$ we have $\m ( \{\theta : \lvert g(\theta) \rvert > sD \}) \leq \frac{1}{s^2 D}$.
\end{lemma}

\begin{lemma}[Borel-Cantelli]\label{lem:BC}
Let $B_1, B_2, \dots$ be  measurable subsets in a measure space $(X,\m)$ with $\sum_{i=1}^\infty \m (B_i) < \infty$. Then $$\m \Biggl(\bigcap_{i=1}^{\infty} \bigcup_{n=i}^{\infty} B_n\Biggr) = 0.$$
\end{lemma}

\begin{proof}[Proof of lemma~\ref{lem:limit0}]
We show that for a.e.\ $\varphi \in \R$ and for all $ \varepsilon > 0$: $$\limsup_{N \rightarrow \infty} \Biggl|\frac{1}{N} \sum_{n=0}^{N-1} f_n(\varphi)\Biggr| < \varepsilon.$$
Let $\varepsilon >0$. 
Since $N$ is growing, it is possible to approximate it by a square: For every $N$ we choose $K \in \N$ with $K^2 \leq N < (K+1)^2$. This ensures $N - K^2 \leq 2K \leq 2 \sqrt{N}$ and therefore 
$$\frac{1}{N} \Biggl\lvert \sum_{n=K^2}^{N-1} f_n(\varphi) \Biggr\rvert \leq \frac{1}{N} \bigl(N-1-K^2\bigr) 2 \lVert\varphi\rVert_\infty 
\ll \frac{\sqrt{N}}{N} = \frac{1}{\sqrt{N}}
$$ 
and so $\frac{1}{N} \Bigl\lvert \sum_{n=K^2}^{N-1} f_n(\varphi) \Bigr\rvert < \frac{\varepsilon}{2}
$
by choosing $N$ large enough. So it is enough to prove the following 

\begin{claim}For $K$ big enough and for a.e.\ $\varphi \in [-\frac{1}{2},\frac{1}{2}]$ we have
$$\Biggl\lvert \frac{1}{K^2} \sum_{n=0}^{K^2-1} f_n(\varphi) \Biggr\rvert < \frac{\varepsilon}{2}.$$
\end{claim}

Define the function $g(\theta) = \frac{1}{N} \sum_{n=0}^{N-1} f_n(\theta)$ on $\Omega = I = [-\frac{1}{2},\frac{1}{2}]$.
Using Lemma~\ref{lem:integral} we compute the second moment $D_0 = \lVert g \rVert_2^2$ of $g$:
\begin{align*}
N^2 D_0 =& \int_\Omega N^2 g(\theta)^2 d\theta 
= \int_I \left(\sum_{n=0}^{N-1} f_n(\theta) \right)^2 d\theta \\
%=& \int_I (\sum_{n=0}^{N-1} f_n(\theta) )^2 d\theta \\
=& \int_I \sum_{n=0}^{N-1} f_n(\theta) ^2 d\theta +
    2 \int_I \sum_{0 \leq m < n < N} f_m(\theta) f_n(\theta)  d\theta \\
=& \sum_{n=0}^{N-1} \int_I f_n(\theta) ^2 d\theta +
    2 \sum_{0 \leq m < n < N} \int_I f_m(\theta) f_n(\theta)  d\theta \\
\ll& N \int_I \lVert\varphi\rVert_\infty ^2 d\theta +
    \sum_{0 \leq m < n < N} \lambda^{-\frac{\lvert n-m \rvert}{2}}  \\
\ll& N +
    N \sum_{j = 1}^{N-1} \lambda^{-\frac{j}{2}} 
\ll N
\end{align*}
so that we get $D_0 \leq \tfrac{C}{N}$ for some constant $C$. Choosing $D=\tfrac{C}{N}$ and $s= \frac{\varepsilon N}{2 C}$ we have $sD = \frac{\varepsilon}{2}$. We define the sets
$A_N = \{ \varphi \in [-\tfrac{1}{2},\tfrac{1}{2}] \colon \bigl\lvert \tfrac{1}{N} \sum_{n=0}^{N -1} f_n(\varphi) \bigr\rvert > \frac{\varepsilon}{2} \}$
 and apply Lemma~\ref{Chebyshev} to get $\m (A_N) \leq \frac{1}{s^2 D} = \frac{4 C}{\varepsilon^2 N}$.

Next we will apply Lemma~\ref{lem:BC} (Borel-Cantelli) to the sets $B_K = A_{K^2}$ with $\m $ the Lebesque measure. Since
$$\sum_{K=1}^\infty \m (B_K) = \sum_{K=1}^\infty \m (A_{K^2}) \leq \sum_{K=1}^\infty \frac{4 C}{\varepsilon^2 K^2} =  \frac{4 C}{\varepsilon^2} \frac{\pi^2}{6} < \infty$$
we have $\m (\bigcap_{i=1}^{\infty} \bigcup_{K=i}^{\infty} A_{K^2}) = 0$. Now we analyse this set: $\bigcup_{K=i}^{\infty} A_{K^2} = \{ \varphi \in I\colon $ $ \exists K \geq i \colon \lvert \frac{1}{K^2} \sum_{n=0}^{K^2-1} f_n(\varphi) \rvert > \frac{\varepsilon}{2} \}$ and $\bigcap_{i=1}^{\infty} \bigcup_{K=i}^{\infty} A_{K^2} = \{\varphi \in I\colon \forall i \in \N ~ \exists K \geq i \colon$ $ \lvert \frac{1}{K^2} \sum_{n=0}^{K^2-1} f_n(\varphi) \rvert > \frac{\varepsilon}{2} \} \supseteq \{\varphi \in I\colon \limsup_{K \rightarrow \infty} \lvert \frac{1}{K^2} \sum_{n=0}^{K^2-1} f_n(\varphi) \rvert > \frac{\varepsilon}{2} \}$. By Borel-Cantelli the measure of the last set has to be zero and we know that for almost every $\varphi \in I$ $\limsup_{K \rightarrow \infty} \lvert \frac{1}{K^2} \sum_{n=0}^{K^2-1} f_n(\varphi) \rvert \leq \frac{\varepsilon}{2}$ holds. This proves the above claim and therefore the Lemma itself.
\end{proof}

\begin{proof}[Proof of Proposition~\ref{prop:invariance}]
Choose a countable family of functions in $\Lip(X)$ that are dense in $C(X)$. Applying Lemma~\ref{lem:limit0} to each of these gives us a conull set $P \subseteq [-\tfrac 12,\tfrac 12]$ such that the conclusion of the lemma holds for all $\varphi \in P$. Recall that this shows \eqref{mainconvergence} for those~$\phi$. 

If now~$\mu$ is a weak$^*$ limit of the sequence of measures as in \eqref{muN}, 
\eqref{mainconvergence} shows $\int \phi((\mathbbm{1},v_a)y)d\mu = \int \phi(y) d\mu$. 
Using density of the set of functions $\phi$ we have chosen above the proposition follows.
\end{proof}

\subsection{Using Poincar\'e Recurrence}
Using the invariance under $V$ from Proposition~\ref{prop:invariance} we can prove tho following theorem:

\begin{theorem} \label{thm:poincare}
Let $W = \{w_\varphi\}$ as before %with $W \cap V_S^{0-} = \{0\}$ 
and $x \in \T^d$. Then for almost every $\varphi \in [-\tfrac{1}{2}, \tfrac{1}{2}] \colon x + w_\varphi \in Eq(S)$.
\end{theorem}

Up to here we have always written elements in $K\ltimes\R^d$ in the coordinates 
$(k,x)$ corresponding to $\begin{pmatrix}
k & x\\
0 & 1
\end{pmatrix} $ (which simplified the action of~$\R^d$ on the left). 
Equivalently we can use the coordinates $[k,x]$
corresponding to $\begin{pmatrix}
k & 0\\
0 & 1
\end{pmatrix}\begin{pmatrix}
\mathbbm{1} & x\\
0 & 1
\end{pmatrix}$, which are more convenient to use
for the following argument. 
This gives the transformation laws $[k,x] = (k,kx)$, $(k,x) = [k, k^{-1} x]$  
between these coordinate systems for all~$k\in K$ and $x\in\R^d$. 
The multiplication rule has now the form $[k_0, x] [k,y] = [k_0 k, k^{-1}x +y]$ and the 
transformation on $X$ (defined by $\tilde{S} (k,x) \Gamma= (k_S k, S_{\mathrm{pos}}x)\Gamma$
for all~$(k,x)\in K\ltimes\R^d$)
now becomes
$$
\tilde{S} [k,y] \Gamma= \tilde{S} (k,ky) \Gamma= (k_S k, S k_S k y) \Gamma= [k_S k, Sy]\Gamma
$$
in new coordinates $[k,y]\in K\ltimes\R^d$.

From now on we work in the new coordinates, which have the advantage that
\begin{equation}\label{isomorph}
	 [k,x] \Gamma \in X\mapsto [k,x+\Z^d]\in K\times \T^d.
\end{equation}
is an isomorphism. For convenience of notation we will use this isomorphism
implicitly and write~$[k,x]\in X$ if~$k\in K$ and~$x\in\T^d$. We also note that,
with this understanding,  the factor map 
$\Theta$ is now defined by $\Theta([k,x]) = x$
for all~$[k,x]\in X$.

Proposition~\ref{prop:invariance} gives us many~$\tilde{S}$-invariant measures
on~$X$ that are also
invariant under all $v \in V$. In the new coordinates this gives invariance under the transformation 
\[
 [k,x] \mapsto [\mathbbm{1}, v] \cdot [k,y] = [k, y + k^{-1} v]
\]
 for any~$v\in V$.

\begin{definition}
 A subspace $P$ is called rational if it can be written as the linear span of rational vectors: $P=\{v=\sum_{j=1}^k \kappa_j e_j \colon \kappa_j \in \R \}$ for some $1 \leq k \leq d \text{ and } e_j \in \Q^d$.
\end{definition}

We recall that $P$ is rational if and only if $P + \Z^d$ is closed in $\T^d$, and all connected subgroups of $\T^d$ are of this form.

\begin{proof}[Proof of Theorem~\ref{thm:poincare}]
Recall that $W \nsubseteq V_S^{0-}$ is a one-dimensional subspace with $V$ its dominating eigenspace
with respect to~$S_{\mathrm{pos}}$. As before fix some $x \in X$ and let 
$\mu = \mu_{x, \varphi}$ be a weak* limit of 
$\mu_N = \tfrac{1}{N} \sum_{n=0}^{N-1} \tilde{S}^n_* \delta_{(\mathbbm{1},x + w_\varphi)\Gamma}$. 
We choose~$\varphi\in[-\frac12,\frac12]$ so that it satisfies the conclusion of Proposition~\ref{prop:invariance} and obtain that  $\mu$  is invariant under
$\tilde{S}$ and 
the left action of $V$.  

In particular, we may consider the ergodic decomposition of~$\mu$ with respect to the action of~$V$. 
As is well known we may obtain the ergodic components of~$\mu$ using the probability space~$(X,\mu)$ 
itself (see e.g.\ \cite[Theorem 6.2]{EW}).
For this let $\mathcal{E} = \{B\in\mathcal B_X\mid B \text{ is invariant under } V\}$ 
and decompose  $\mu$ into conditional measures to obtain the
decomposition into $V$-ergodic components: 
$\mu = \int \mu_{[k,y]}^{\mathcal{E}} d\mu([k,y])$. 
We will show that almost every ergodic component will be the 
Lebesgue measure on the fiber $\{k\} \times \T^d$.

We fix a typical $[k,y]\in X$ together with its ergodic measure $\mu_{[k,y]}^{\mathcal{E}}$
and may assume (see \cite[Thm.~6.2, Thm.~8.20]{EW})
that the point~$[k,y]$ indeed equidistributes w.r.t.\ the action of~$V$
to its ergodic component $\mu_{[k,y]}^{\mathcal{E}}$. 
As noted before the action of $v \in V$ has the form
$$
 [\mathbbm{1},v] [k,y] = [k, k^{-1}v + y]
$$
for any $v \in V$,~$k\in K$, and~$y\in\T^d$, 
i.e.\ the coordinate~$k\in K$ remains unchanged and we have simply the translation action
of~$k^{-1}V$ on~$\T^d$. It follows that $\mu_{[k,y]}^{\mathcal{E}}$ is the Lebesgue measure on 
the closure of the connected group $k^{-1}V + \Z^d$ within $\T^d$.  Recall
that $\overline{k^{-1}V + \Z^d} = P_k + \Z^d$ for a rational subspace $P_k$, which only depends on $k \in K$ (and not on $y \in \T^d$).
Hence $\mu_{[k,y]}^{\mathcal{E}}$ is the normalized Lebesgue measure $m_{[k,P_k + y]}$
supported on some affine rational subspace $[k,P_k+y]$.

Therefore we define the following map $\Phi \colon X \longmapsto \{ \text{rational subspaces of } \R^d\}$. For any $[k,y] \in X$ let $P_k$ be the rational subspace $\R^d$ such that $\overline{y + k^{-1} V + \Z^d} = y + P_k + \Z^d$ and define $\Phi ([k,y]) = P_k$. 
For a given rational subspace $P$ of $\R^d$ 
we also define the level set $A_P = \{ [k,y] \colon \Phi([k,y]) = P \} \subseteq X$. 
Since the set of rational subspaces is countable we can write $X = \bigsqcup_{P } A_P$ as a countable union of sets of this form.

Now we want to analyse what happens when we act with $\tilde{S}$:
Since $V$ is an eigenspace of $S_{\mathrm{pos}}$, it follows that the~$\sigma$-algebra of~$V$-invariant
sets is invariant under~$\tilde{S}$. It follows that the ergodic components for the action of~$V$
are almost surely mapped to the ergodic components under the action of~$\tilde{S}$, i.e.~$\tilde{S}_*\mu_{[k,y]}^{\mathcal E}=\mu_{\tilde{S}[k,y]}^{\mathcal E}$. 
However, $\mu_{[k,y]}^{\mathcal E}=m_{[k,P_k+y]}$ and~$\mu_{\tilde{S}[k,y]}^{\mathcal E}=
\mu_{[k_sk,Sy]}=m_{[k_sk,P_{k_sk}+Sy]}$ almost surely and so~$S(P_k)=P_{k_sk}$.

Choose $P$ to be a rational subspace such that $\mu (A_{P}) > 0$. 
Applying Poincar\'e recurrence tells us that there exists an $l \in \N$ such that $\mu (A_{P} \cap \tilde{S}^l A_{P}) > 0$. Together with the above this implies that~$S^l P=P$. However, since $S$ is totally irreducible and $\dim P \geq 1$, we know that $P = \R^d$ (for more details regarding this fact see Section~\ref{dim1}).
This means that only one level set~$A_{P}$ has positive measure, namely the one for~$P=\R^d$, 
and hence every ergodic component $\mu_{[k,y]}^{\mathcal{E}}$ is the Lebesgue measure on the fiber $\{k\} \times \T^d$. 

Now we take the push forward of $\mu_N$ and $\mu$ under the factor map $\Theta: X \longrightarrow \T^d$ with $[k,x] \mapsto x$: It is easy to see that $\Theta_* \mu_N = \tfrac{1}{N} \sum_{n=0}^{N-1} S^n_* \delta_{x + w_\varphi}$ and $\Theta_* \mu$ is the Lebesgue measure on the torus. Of course we have $\Theta_* \mu_N \rightarrow \Theta_*\mu =  \m_{\T^d}$ as $N \rightarrow \infty$ which immediately gives $x+w_\varphi \in Eq(S)$.
\end{proof}

%Here we prove the statement following from fullmeasure on a line
\subsection{Proof of Theorem~\ref{thm:meas}}

\begin{lemma}\label{lem:winproj}
 Let $W \subseteq V_T^{0-}$ be a 1-dimensional subspace and let ${W^{\bot}}$ be any subspace of $\R^d$ with $W \oplus {W^{\bot}} = \R^d$. Then $(x_1+(B(1,0)\cap {W^{\bot}}))\cap ND(T)$ is thick inside $(x_1+(B(1,0)\cap {W^{\bot}}))$.
\end{lemma}
\begin{proof}
Recall that we already showed that~$ND(T)$ is thick. Now the lemma 
follows immediately from Lemma~\ref{lem:foliation}
and the properties of the Hausdorff dimension for products of the form~$W_1\times(B(1,0)\cap W)$
with~$W_1\subset W^\bot$.
\end{proof}

\begin{proof}[Proof of Theorem~\ref{thm:meas}]
The proof follows the same line as the proof of Theorem~\ref{thm:win}, we have only to exchange the roles of $(S,Eq(S))$ and $(T,ND(T))$. Of course we define $W \subseteq V_T^{0-}$ such that $W \cap V_S^{0-} = \{0\}$ and use the norm defined in \S \ref{sec:ChaikaEskin}. 
We apply Lemma \ref{lem:winproj} to show that the set
$\{v\in B(0,1)\cap W^\bot \vert x_1+v\in ND(T) \}$
 is thick and later Theorem \ref{thm:poincare} to ensure that $Eq(S) \cap A_0(x)$ has full measure as a subset of $A_0(x)$. The Marstrand Slicing Theorem \ref{thm:MarSlice} concludes the proof.
\end{proof}

%Here we describe what kind of algebraic properties S and T fullfill if we cannot apply section 3 and 4
\section{Equality of weak stable subspaces}

In this section we analyze the case $V_S^{0-} = V_T^{0-}$ and prove Theorem~\ref{conjecture}.

\subsection{The simpler case $\dim V_S^{0-} = 1$} \label{dim1}
Let $S$ be totally irreducible and let $W$ be a rational subspace invariant under $S$. 
Then we claim that $W$ has to be either~$\{0\}$ or~$\R^d$. In fact, this follows since
the characteristic polynomial
of~$S$ restricted to~$W$ has rational coefficients and is a divisor or the characteristic polynomial
of~$S$. Using this and Galois theory we prove the following first step towards Theorem~\ref{conjecture}.

\begin{lemma}\label{firstlemma}
Let $S$ and $T$ be irreducible and $v$ be a common eigenvector of $S$ and $T$. Then $S$ and $T$ commute.
\end{lemma}

\begin{proof}
Denote with $K_S$ and $K_T$ the field extensions of $\Q$ such that the characteristic polynomials of $S$ and $T$ are split. With $K$ we denote the smallest common field extension of $K_S$ and $K_T$. 

Let $v$ be the common eigenvector with its respective eigenvalues $\lambda_S$ and $\lambda_T$ satisfying $Sv = \lambda_S v$ and $Tv = \lambda_T v$. We may also choose the eigenvector $v$ to be algebraic with\footnote{In fact it is not hard to see that one can choose~$v\in(K_S\cap K_T)^d$.} $v \in  K^d$. 

The field extension $K\mid\Q$ is algebraic and Galois. We denote with $\sigma_1, \dots, \sigma_n$ the Galois automorphisms of $K\mid\Q$.

Each $\sigma_i$ we can apply to the equation $Sv = \lambda_S v$ getting $S \sigma_i(v) = \sigma_i(\lambda_S) \sigma_i(v)$ and $T \sigma_i(v) = \sigma_i(\lambda_T) \sigma_i(v)$. Of course $S,T$ have integer entries and remain unchanged. So we know that $\sigma_i(v)$ is a common eigenvector of $S$ and $T$ for every $1 \leq i \leq n$. 

We define $W = \text{span} \{ \sigma_1 (v), \dots, \sigma_n (v) \}$. When we apply $\sigma_i$ to $W$, those vectors get permuted, so $W$ is invariant under all Galois automorphisms. This is only possible if $W$ is a rational subspace. Every $\sigma_i (v)$ is an eigenvector of $S$, so that we know that $W$ must be invariant under $S$. Now $S$ is totally irreducible and $W$ rational, therefore $W = \R^d$. 
So we can choose a basis out of the generating set $\{v, \sigma_1 (v), \dots, \sigma_n (v) \}$ which is a basis of $\R^d$ in which $S$ and $T$ are both simultaneously diagonal and therefore commute.
\end{proof}

We note that the above already proves some cases of Theorem~\ref{conjecture}.

\subsection{The general case}
Now we assume $V_S^{0-} = V_T^{0-}$ and $\dim V_S^{0-} > 1$. We will 
reduce this case to the case above.

As in the above lemma
we do not care in the following algebraic argument that those subspaces are weak stable.
The only important assumption is that we have a nontrivial subspace which is simultaneously invariant under $S$ and $T$.

\begin{lemma}\label{lemmadivisible}
Let~$S$ and~$T$ be two totally irreducible integer matrices.
Among the nontrivial subspaces simultaneously invariant under $S$ and $T$ there is one with a minimal dimension $p$, and the dimension of all other such subspaces is divisible by $p$.
\end{lemma}

\begin{proof}
For this proof, denote with $\mathcal{I}$ the set of subspaces of $\overline{\Q}^d$ which are invariant under $S$ and $T$. This set is closed under taking intersections and sums. 
Let $V_0 \in \mathcal{I}$ be such that $\dim (V_0) = p>0$ is minimal. 

Let~$K\mid\Q$ be a Galois field extension over which both~$S$ and~$T$ are diagonalizable and 
let~$\{\sigma_1,\ldots,\sigma_n\}$ be the Galois group of~$K$ over~$\Q$. 
These automorphisms map elements of $\mathcal{I}$ to elements of $\mathcal{I}$ of the same dimension. 
Now consider $\sigma_i ( V_0)$ and notice that either  $\sigma_i ( V_0)=V_0$, or $\sigma_i ( V_0)$ will intersect $V_0$ trivially. 

The sum $W=\sigma_1 (V_0) + \dots + \sigma_n (V_0)$ is invariant under all Galois automorphisms, hence a rational subspace,  and clearly invariant unter the totally irreducible $S$, hence it can only be $\R^d$. Adding these subspaces step by step, in each step either $\sigma_i (V_0)$ is contained in the sum of the preceeding subspaces, or it intersects the previous sum trivially and the dimension of the sum increases exactly by $p$.
Therefore~$p\mid d$.

Next we start with an arbitrary subspace $V\in\mathcal{I}$ of dimension $q$. As before we add consecutively the subspaces $\sigma_1 (V_0), \dots, \sigma_n (V_0)$ to it. Again in every step minimality of~$V_0$
implies that either $\sigma_k(V_0)$ is already contained in the previous sum or it intersects it trivially. Hence we see by induction that for all~$1\leq k\leq n$
\[
 V+\sigma_1(V_0)+\cdots+\sigma_k(V_0)=V\oplus\textstyle\bigoplus_j'\sigma_j(V_0),
\]
where $\bigoplus_j'$ denotes the direct sum over some of the indices~$j\in\{1,\ldots,k\}$.

For $k=n$ this implies that
\[
 d=\dim\Q^d=\dim\left(V+\sigma_1(V_0)+\cdots+\sigma_n(V_0)\right)
=q+\ell p
\]
for some~$\ell\leq n$. Since~$p\mid d$ this gives the lemma.
\end{proof}

\begin{proof}[Proof of Theorem~\ref{conjecture}]
  Let $S,T$ be totally irreducible automorphisms of $\T^d$ satisfying $V_{S}^{0-} = V_{T}^{0-}$ and
 $\gcd(d, \dim V_S^{0-}) = 1$. By Lemma~\ref{lemmadivisible} there exists a common eigenvector
for~$T$ and~$S$, and by Lemma~\ref{firstlemma} the maps $T$ and $S$ commute.
\end{proof}

\subsection{An example}

What happens when $\min_{V \in I \backslash \{0\}} \dim V >  1$? We now give an example to show that in this case there exist $S$ and $T$ which do not commute but are nonetheless algebraically related.

Let $d = 4 = 2 \cdot 2$. We construct an example acting on the space $\Q^4 = \Q^2 \otimes \Q^2$.

Let $A = \begin{pmatrix} 2 & 3 \\ 1 & 2  \end{pmatrix}$ and $B = \begin{pmatrix} 2 & 5 \\ 1 & 2  \end{pmatrix} \in \GL_2(\Z)$ with eigenvalues $2 \pm \sqrt{3}$ and $2 \pm \sqrt{5}$. We note that $A$ and $B$ do not commute and that $K_A = \Q(\sqrt{3})$ and $K_B = \Q(\sqrt{5})$ are the field extensions such that the characteristic polynomials split into completely. 

We also consider the linear map 
$Q=\begin{pmatrix} 1 & 2 \\ 1 & 1 \end{pmatrix}$ 
acting on $\Q^2$ with eigenvalues $1 \pm \sqrt{2}$. Let $\lambda_{A,i}$, $\lambda_{B,i}$ denote the eigenvalues of $A$ and $B$ for~$i=1,2$. Now we choose $k \in \N$ such that $\lambda_{A,i}  (1+\sqrt{2})^k > 1, \lambda_{B,i}  (1+\sqrt{2})^k > 1, \lambda_{A,i}  (1-\sqrt{2})^k < 1$ and $\lambda_{B,i}  (1-\sqrt{2})^k < 1$ for~$i=1,2$. 
A short calculation shows that we can take $k=2$ and calculate
$Q^2=\begin{pmatrix} 3 & 4 \\ 2 & 3  \end{pmatrix}$.
Then we define $S = A\otimes Q^2$ and $T = B \otimes Q^2$ which we identify with the $4$-by-$4$ matrices 
$$S = \begin{pmatrix} 3A & 4A \\ 2A & 3A  \end{pmatrix} = \begin{pmatrix} 6 & 9 & 8 & 12 \\ 3 & 6 & 4 & 8 \\ 4 & 6 & 6 & 9 \\ 2 & 4 & 3 & 6  \end{pmatrix}\text{ and } T = \begin{pmatrix} 3B & 4B \\ 2B & 3B  \end{pmatrix} = \begin{pmatrix} 6 & 15 & 8 & 20 \\ 3 & 6 & 4 & 8 \\ 4 & 10 & 6 & 15 \\ 2 & 4 & 3 & 6  \end{pmatrix}.$$

We note that the eigenvalues of~$S$ are~$(2\pm\sqrt{3})(1\pm\sqrt{2})^2$ and of~$T$ are~$(2\pm\sqrt{5})(1\pm\sqrt{2})^2$, which implies that both~$S$ and~$T$ are totally irreducible.

Next note that
the eigenvectors for~$S$ are given by
\begin{align*}
 v^+_{1,2} &= \begin{pmatrix} \pm \sqrt{6} \\ \sqrt{2} \\ \pm \sqrt{3} \\ 1 \end{pmatrix} \text{ (expanding) and }\\
 v^-_{3,4} &= \begin{pmatrix} \mp \sqrt{6} \\ -\sqrt{2} \\ \pm \sqrt{3} \\ 1 \end{pmatrix} \text{  (contracting).}
\end{align*}

For $T$ the eigenvectors are given by a similar expression (replacing~$\sqrt{3}$ by~$\sqrt{5}$)
and from this one can easily see 
\[
 V_S^{0-} = V_T^{0-} = \operatorname{span} \left(\begin{pmatrix} -\sqrt{2} \\ 0 \\ 1 \\ 0 \end{pmatrix}, \begin{pmatrix} 0 \\ -\sqrt{2} \\ 0 \\ 1 \end{pmatrix}\right).
\]
Finally we note that~$S=A\otimes Q^2$ and~$T=B\otimes Q^2$ do not commute since~$A$ and~$B$ do not commute.

This shows that Theorem~\ref{conjecture} cannot hold in full generality (i.e.\ without the condition $\gcd(d, \dim V_S^{0-}) = 1$ or something similar).
Instead of commutativity of the maps we have another   quite strong algebraic condition: If we compare the two field extensions $K_S$ and $K_T$ of $\Q$ such that the characteristic polynomials of $S$ and $T$ are split, we have $K_S = \Q(\sqrt{2}, \sqrt{3})$ and $K_T = \Q(\sqrt{2}, \sqrt{5})$, in particular $K_S \cap K_T = \Q(\sqrt{2}) \supsetneq \Q$.

\subsection{Concluding remarks}
To summarize we have shown in the case where the weak stable subspaces of $S$ and $T$ are not identical 
that $\dim (Eq(S) \cap ND(T)) = d$. On the other hand if the weak stable subspaces are identical
but their dimension is coprime to $d$ we have that $S$ and $T$ commute. In this case we can refer to the work of Bergelson, the first named author and Tseng \cite{BET}, which gives a related but weaker conclusion assuming (as is necessary)
that $T$ and $S$ are multiplicatively independent. The above example, however, fits with neither of the two
settings and it would be interesting to see how to extend the argument to that case.


\begin{thebibliography}{99}
\bibitem{BET} V. Bergelson, M. Einsiedler, J. Tseng, \emph{Simultaneous dense and nondense orbits for commuting maps}, Preprint, arXiv:1309.4823v1.

%\bibitem{BFK} R. Broderick, L. Fishman, D. Kleinbock, \emph{Schmidt's game, fractals, and orbits of toral endomorphisms}, Ergodic Theory Dynam. Systems 31 (2011), no. \textbf{4}, 1095–1107. 

%\bibitem{D1} S. G. Dani, \emph{On badly approximable numbers, Schmidt games and bounded orbits of flows}, M. m. Dodson and J. A. G. Vickers (eds), \emph{Number theory and dynamical systems}, Lonton Mathematical Society Lecture Note Series \textbf{134}, Cambridge University Press, Cambridge, UK (1989).

%\bibitem{D2} \bysame, \emph{On orbits of endomorphisms of tori and the Schmidt game}, Erdogic Theory Dunam. Systems \textbf{8} (1988), 523-529.

%\bibitem{EKL} M. Einsiedler, A. Katok, and E. Lindenstrauss, \emph{Invariant measures and the set of exceptions to Littlewood’s conjecture}, Ann. of Math. (2) \textbf{164} (2006), no. 2, 513–560.

%\bibitem{EL} M. Einsiedler and E. Lindenstrauss, \emph{Rigidity properties of Zd-actions on tori and solenoids}, Electron. Res. Announc. Amer. Math. Soc. \textbf{9} (2003), 99–110.

\bibitem{EW} M. Einsiedler and W. Ward, \emph{Ergodic Theory with a view towards Number Theory}, Springer GTM 259  (2011).


%\bibitem{F} D. F\"arm, \emph{Simultaneously non-dense orbits under different expanding maps}, Dyn. Syst. \textbf{25} (2010), 531–545.

\bibitem{F1} L. Fishman, \textit{Schmidt's games, badly approximable matrices, and fractals}, J. Number Th. 129(2009), 2133-2153.

\bibitem{F2}\bysame, \textit{Schmidt's game on fractals}, Israel J. Math. 171 (2009), 77-92.

%\bibitem{Fo} G. B. Folland, \emph{Real analysis : modern techniques and their applications}, Pure and applied mathematics, Wiley, New York, 1984. 

%\bibitem{G} E. Glasner, \emph{Ergodic theory via joinings}, Mathematical Surveys and Monographs, vol. \textbf{101}, 2003.

%\bibitem{K} D. Kleinbock, \emph{Nondense orbits of flows on homogeneous spaces}, Ergodic Theory Dynam. Systems \textbf{18} (1998), 373-396.

\bibitem{KM} D. Y. Kleinbock, G. A. Margulis, \emph{Bounded orbits of nonquasiunipotent flows on homogeneous spaces}, Amer. Math. Soc. Transl. (2) \textbf{171} (1996), 141-172.  

%\bibitem{KW1} D. Kleinbock and B. Weiss, \emph{Modified Schmidt games and Diophantine approximation with weights}, Adv. Math. \textbf{223} (2010), 1276–1298.


%\bibitem{KW2} D. Kleinbock, B. Weiss, \emph{Modified Schmidt games and a conjecture of Margulis}, Preprint, arXiv:1001.5017v2.

%\bibitem{L} D. A. Lind, \emph{Dynamical properties of quasihyperbolic toral automorphisms}, Ergodic Theory Dynamical Systems 2 (1982), no. \textbf{1}, 49–68.

\bibitem{LM} B. Lytle and A. Maier, \emph{Simultaneous dense and nondense orbits for noncommuting toral endomorphisms}, Preprint, arXiv:1404.4014.

\bibitem{CE} J. Chaika and A. Eskin, \emph{Every flat surface is Birkhoff and Oseledets generic in almost every direction}, Preprint, arXiv:1305.1104.

\bibitem{S} W. M. Schmidt, \emph{On badly approximable numbers and certain games}, Trans. A.M.S. \textbf{123} (1966), 27-50.

%\bibitem{W} P. Walters, \emph{An Introduction to Ergodic Theory}, Graduate Texts in Mathematics, vol \textbf{79}, Springer, New York, 1982. 

%\bibitem{G} \bysame, \emph{Ergodic theory via joinings}, Mathematical Surveys and Monographs, vol. 101, 2003.
\end{thebibliography}
\end{document}